\documentclass{amsart}

\usepackage{amssymb, amsmath}
\usepackage{mathrsfs}
\usepackage{amscd}
\usepackage{verbatim}
\usepackage{enumerate}

\usepackage{hyperref}

\allowdisplaybreaks

\theoremstyle{plain}
\newtheorem{theorem}{Theorem}[section]
\theoremstyle{remark}
\newtheorem{remark}[theorem]{Remark}
\newtheorem{facts}[theorem]{Facts}
\newtheorem{example}[theorem]{Example}
\newtheorem{problem}[theorem]{\textbf{\em Problem}}
\theoremstyle{plain}
\newtheorem{corollary}[theorem]{Corollary}
\newtheorem{lemma}[theorem]{Lemma}
\newtheorem{proposition}[theorem]{Proposition}

\numberwithin{equation}{section}

\def\N{{\mathbb N}}
\def\Z{{\mathbb Z}}

\def\R{{\mathbb R}}
\def\C{{\mathbb C}}

\newcommand{\E}{{\mathbb E}}
\renewcommand{\P}{{\mathbb P}}

\newcommand{\F}{{\mathcal F}}

\newcommand{\g}{\gamma}

\newcommand{\eps}{\varepsilon}

\renewcommand{\O}{\Omega}

\newcommand{\calL}{{\mathcal L}}

\newcommand{\one}{{{\bf 1}}}

\newcommand{\lb}{\langle}
\newcommand{\rb}{\rangle}

\begin{document}

\author{Mark Veraar}
\address{Delft Institute of Applied Mathematics\\
Delft University of Technology \\ P.O. Box 5031\\ 2600 GA Delft\\The
Netherlands} \email{M.C.Veraar@tudelft.nl}

\author{Lutz Weis}
\address{Institut f\"ur Analysis \\
Karlsruhe Institute of Technology\\
D-76128  Karls\-ruhe\\Germany}
\email{Lutz.Weis@kit.edu}

\thanks{The first author
is supported by the VIDI subsidy 639.032.427 of the Netherlands Organisation for Scientific Research (NWO)}

\date\today

\title[Estimates for vector-valued holomorphic functions]
{Estimates for vector-valued holomorphic functions and Littlewood-Paley-Stein theory}

\begin{abstract}
In this paper we consider generalized square function norms of holomorphic functions with values in a Banach space.
One of the main results is a characterization of embeddings of the form
\[L^p(X)\subseteq \gamma(X) \subseteq L^q(X),\]
in terms of the type $p$ and cotype $q$ for the Banach space $X$. As an application we prove $L^p$-estimates for vector-valued Littlewood-Paley-Stein $g$-functions and derive an embedding result for real and complex interpolation spaces under type and cotype conditions.
\end{abstract}




\maketitle

\section{Introduction}
For a space $X = L^r$ with $r\in (1, \infty)$ and a function $f:\Omega\to X$, the usefulness of square functions
\begin{equation}\label{eq:squareintro}
\Big\|\Big(\int_{\Omega} |f(\omega)|^2 \, d\mu(\omega) \Big)^{\frac12}\Big\|_X
\end{equation}
and their estimates, are well-known in
\begin{itemize}
\item harmonic analysis (e.g.\ in the context of Littlewood-Paley theory and $g$-functions \cite{Stein:topics}),
\item in the theory of the holomorphic functional calculus (see e.g. \cite{CDMY96, LeM04})
\item in stochastic analysis (e.g.\ Burkholder-Davis-Gundy inequalities for martingales and stochastic integrals \cite{Kal}).
\end{itemize}
Motivated by the $\ell$-norm in the geometry of Banach spaces (see e.g.\ \cite{FigTom,PisConv}), the paper \cite{KalW04} in an earlier version of 2002, introduced $\gamma$-norms as an extension of the square functions \eqref{eq:squareintro} to the Banach space setting. Since then $\gamma$-norms, similarly as their classical counterpart \eqref{eq:squareintro}, have been very useful in harmonic analysis, spectral theory and stochastic analysis of Banach space valued function. It makes it possible to extend Hilbert space results to the Banach space setting (see e.g.\ \cite{AFM,BCCFRM12,betancor2012umd,HaHa,HaakKunst06,Hyt07,HNP:tent,HytWeBMO,KaiWei08,LeM14,Nee10,NVW07a,NeeWei05a}).

In this paper we consider square functions of  Banach space valued holomorphic functions
which arise naturally in several areas of analysis. In particular, in the evolution equation approach to partial differential equations one typically needs resolvents $\lambda\mapsto  (\lambda-A)^{-1}$ on a sector or strip or holomorphic semigroups $z\mapsto e^{-z A}$ on a sector. Here $A$ is an (unbounded) operator on a Banach space $X$. General references on vector-valued holomorphic functions, functional calculus, and their uses in evolution equations are \cite{Am,ABHN,Haase:2,KunWeis04}.

It is easy to see that on bounded domains for any integer $k\geq 0$, $\sum_{j=1}^k \|f^{(j)}\|_{\gamma}$ is equivalent to any of the Sobolev norms $\|f\|_{W^{\ell,p}}$ for any integer $\ell\geq 0$ and $p\in [1, \infty]$ up to a slight deformation of the domain (see Lemma \ref{lem:analyticbddgamma}). However, for unbounded domains such an equivalence of $\gamma$-norms and Sobolev norms fails as can already be seen in the scalar case. We will show certain embedding results still hold when the domain is a strip or a sector (see Lemmas \ref{lem:analyticgammadiscrte} and \ref{lem:analyticYlemma}). Our main result is then a characterization of an embedding result on a strip or sector $S$ of the form $L^p(X) \hookrightarrow \gamma(X) \hookrightarrow L^q(X)$ in terms of the (Rademacher) type $p$ and cotype $q$ of $X$ (see Theorems \ref{thm:gammaanalytictype} and \ref{thm:gammaanalyticcotype}).

As an application of the embedding result for holomorphic functions on the sector, we consider vector-valued extensions of $L^p$-estimates for Littlewood-Paley-Stein $g$-functions  as introduced in \cite{Stein:topics} in the scalar case. First of all we prove $L^p$-estimates for $t\mapsto f(tA)x$, where $f$ is a bounded holomorphic function on a sector and $A$ has a bounded $H^\infty$-calculus under type and cotype assumptions. Secondly, we consider estimates for the tensor extension of diffusion semigroups on $L^r(\Omega)$. Previous results in this direction have been obtained in \cite{Xu98} for the Poisson semigroup and later in \cite{MTX06} for more general diffusion semigroups. Our approach differs from the latter works and is applicable to a larger class of diffusion operators and $g$-functions.

As a further application we obtain an embedding result for real and complex interpolation spaces in the case $X_1\hookrightarrow X_0$, $X_0$ has type $p$ and cotype $q$ and $X_1 = D(A)$, where $A$ is a sectorial operator with a bounded holomorphic calculus: for all $\theta\in (0,1)$
\begin{align*}
(X_0, X_1)_{\theta,p} \hookrightarrow [X_0, X_1]_{\theta} \hookrightarrow (X_0, X_1)_{\theta,q},
\end{align*}
Here $(\cdot, \cdot)_{\theta, r}$ denotes the real interpolation space with parameter $r$ and $[\cdot, \cdot]_{\theta}$ the complex interpolation space. Of course this embedding always holds for $p=1$ and $q=\infty$. In \cite{Peetre69} this was improved to exponents $p\in (1, 2]$ and $q\in [2, \infty)$ under Fourier type conditions on the Banach space and this leads to a different result than ours.

\medskip

\noindent \textbf{Acknowledgement} The authors thank Nick Lindemulder and the anonymous referee for helpful comments and careful reading.

\section{Preliminaries}

For details on $\gamma$-norms and $\gamma$-radonifying operators we refer to the survey \cite{Nee10}. Below we repeat some of the definitions and properties.

Let $(S,\Sigma,\mu)$ be a measure space and (for convenience) we assume that $L^2(S)$ is separable and let $(h_n)_{n\geq 1}$ be an orthonormal basis for $L^2(S)$. Let $X$ be a (complex) Banach space. Let $(\Omega,\mathcal{A}, \P)$ be a probability space and let $(\gamma_n)_{n\geq 1}$ be a sequence of independent (complex) standard normal random variables with values in $\R$.  We will identity a function $f:S\to X$
and the operator $I_f:L^2(S)\to X$ given by
\begin{equation}\label{eq:Tassf}
I_f h = \int_S f(s) h(s) \, ds
\end{equation}
whenever this integral makes sense as a Bochner or Pettis integral (see \cite{DieUhl}). In particular, the integral
makes sense if $f$ is strongly measurable and for all $x^*\in X^*$, $s\mapsto \lb f(s), x^*\rb$ is in $L^2(S)$. Note that this identification is of a similar nature as the one usually made between functions and distributions.

For a general $T:L^2(S)\to X$ we let
\begin{equation}\label{eq:Tseries}
\|T\|_{\gamma(L^2(S),X)} = \Big\|\sum_{n\geq 1} \gamma_n T h_n\Big\|_{L^2(\Omega;X)}
\end{equation}
whenever this series converges in $L^2(\Omega;X)$. One can check this definition does not depend on the choice of the orthonormal basis $(h_n)_{n\geq 1}$. Moreover,
\[\|T\|_{\calL(L^2(S),X)}\leq \|T\|_{\gamma(L^2(S;X))}.\]
The space of all $T$ for which the series in \eqref{eq:Tseries} converges in $L^2(\Omega;X)$ is denoted by $\gamma(L^2(S),X)$. This can be shown to be a Banach space again. For $f:S\to X$ as above we let $\|f\|_{\gamma(S;X)} =  \|I_f\|_{\gamma(L^2(S),X)}$, where $I_f$ is as in \eqref{eq:Tassf}.
In particular, for $f$ of the form $f(s) = \sum_{n=1}^N h_n(s) x_n$ we find
\[\|f\|_{\gamma(S;X)} = \Big\|\sum_{n=1}^N \gamma_n x_n\Big\|_{L^2(\Omega;X)}.\]
In particular, $\|s\mapsto h(s)x\|_{\gamma(S;X)} = \|h\|_{L^2(S)} \|x\|$ whenever $h\in L^2(S)$ and $x\in X$.

The operators of $\gamma(L^2(S),X)$ satisfy the so-called ideal property. As a consequence of this one can extend
many operations on $L^2(S)$ to $\gamma(L^2(S),X)$.

\begin{facts}
The following properties will be used frequently:
\begin{enumerate}[(a)]
\item For $f\in \g(S;X)$ and $x^*\in X^*$, $\|\lb f, x^*\rb\|_{L^2(S)}\leq \|f\|_{\gamma(S;X)} \|x^*\|$.
\item For $f\in \g(S;X)$ and $g\in L^\infty(S)$, $fg\in \g(S;X)$ and
\[\|g f\|_{\gamma(S;X)}\leq \|g\|_{L^\infty(S)} \|f\|_{\gamma(S;X)}.\]
\item If $f\in \gamma(S;X)$ and $S_0\subset S$, then $\|f\|_{\gamma(S_0;X)} \leq \|f\|_{\gamma(S;X)}$ and if moreover $f$ is supported in $S_0$, then $\|f\|_{\gamma(S_0;X)} = \|f\|_{\gamma(S;X)}$.
\item For $f\in \gamma(\R;X)$ and $h\in \R$, one has $\|x\mapsto f(x+h)\|_{\gamma(\R;X)} = \|f\|_{\gamma(\R;X)}$.
\item For $f\in \gamma(\R;X)$ and $a>0$, one has $\|x\mapsto f(ax)\|_{\gamma(\R;X)} = a^{-1/2} \|f\|_{\gamma(\R;X)}$.
\item The Fourier transform
\[\hat{f}(\xi) = \F(f)(\xi) = \int_{\R} e^{-2\pi i x \xi} f(x) \, dx, \ \ \ \xi\in \R\]
satisfies $\|\hat{f}\|_{\gamma(\R;X)} = \|f\|_{\gamma(\R;X)}$.
\item For $g\in L^1(\R)$ and $f\in \gamma(\R;X)$, $g*f\in \gamma(\R;X)$ and
\[\|g*f\|_{\gamma(\R;X)}\leq \|g\|_{L^1(\R)} \|f\|_{\gamma(\R;X)}.\]
\end{enumerate}
\end{facts}

Let $I \subset \R$ be a finite or infinite interval. If $f\in \gamma(I;X)$ is such that for every $j\leq k$, the $j$-th derivatives exist in the distributional sense and $f^{(j)}\in \gamma(I;X)$, then we write $f\in \gamma^{k}(I;X)$ and let
\[\|f\|_{\gamma^{k}(I;X)} = \sum_{j=0}^{k}\|f^{(j)}\|_{\gamma(I;X)}.\]

\section{General results for holomorphic functions and $\gamma$-norms}

We start with a lemma for bounded domains.
\begin{lemma}\label{lem:analyticbddgamma}
Let $D\subset \C$ be open and assume $f:D\to X$ is holomorphic. Let $D_1, D_2, D_3\subseteq D$ be open and bounded and such that $\overline{D_1}\subseteq D_2$ and $\overline{D_2}\subseteq D_3$. Then for all $p\in [1, \infty]$ and all integers $k, \ell\geq 0$, the following estimates hold:
\begin{equation}\label{eq:estfWgamma}
\|f\|_{W^{\ell,p}(D_1;X)} \lesssim \|f\|_{\gamma^k(D_2;X)} \lesssim \|f\|_{W^{\ell,p}(D_3;X)}.
\end{equation}
\end{lemma}
The constants in \eqref{eq:estfWgamma} are independent of $f$ and we will not explicitly write this in the sequel.

\begin{proof}
To prove the second estimate in \eqref{eq:estfWgamma} it suffices to take $\ell=0$ and $p=1$. Note that we can cover $D_2$ by finitely many balls contained in $D_3$. By a dilation and translation argument, we may assume that $D_2 = \{z: |z|<1\}$ and $D_3 = \{z: |z|<1+2\varepsilon\}$ for some $\varepsilon>0$.

By Cauchy's formula we can write
\[f^{(n)}(a) = \frac{n!}{2\pi i} \oint_{\{|z| = t\}} \frac{f(z)}{(z-a)^{n+1}}\, dz, \ \ \ |a|<1+\varepsilon,\]
where $t\in [1+\varepsilon,1+2\varepsilon]$ is fixed. Therefore,
\[\|f^{(n)}(0)\| \leq \frac{n!}{2\pi t^{n+1}} \oint_{\{|z| = t\}} \|f(z)\|\, |dz|.\]
Integrating over $t\in [1+\varepsilon, 1+2\varepsilon]$, we find
\begin{equation}\label{eq:estfnaccent}
\|f^{(n)}(0)\| \leq C_{\varepsilon} \frac{n!}{(1+\varepsilon)^n} \|f\|_{L^1(D_3)}.
\end{equation}

Writing $f(z) = \sum_{n=0}^\infty \frac{z^n}{n!} f^{(n)}(0)$, by the triangle inequality and \eqref{eq:estfnaccent}  we find
\begin{align*}
\|f^{(k)}\|_{\gamma(D_2;X)} & \leq  \sum_{n=k}^\infty \frac{\|z^{n-k}\|_{L^2(D_2)}}{(n-k)!} \|f^{(n)}(0)\| \\ & \leq
C_{\varepsilon} \|f\|_{L^1(D_3)}\sum_{n=k}^\infty \frac{n!}{(n-k)!} (1+\varepsilon)^{-n}
= C_{\varepsilon,k} \|f\|_{L^1(D_3)}.
\end{align*}

To prove the first estimate in \eqref{eq:estfWgamma} it suffices to consider
$D_1 = \{z: |z|<1\}$ and $D_2 = \{z: |z|<1+2\varepsilon\}$ for some $\varepsilon>0$. Let $A_{\varepsilon} = \{1+\varepsilon<|z|<1+2\varepsilon\}$.
Using Cauchy's formula again we find that for all $|a|\leq 1$ and all $x^*\in X^*$,
\begin{align*}
\varepsilon |\lb f^{(k)}(a), x^*\rb| & = \frac{k!}{2\pi} \int_{1+\varepsilon}^{1+2\varepsilon} \oint_{\{|z| = t\}} \frac{|\lb f(z), x^*\rb|}{|z-a|^{k+1}}\, |dz| \, dt
\\ & \leq \frac{k! C_{\varepsilon}}{2\pi} \|\lb f, x^*\rb\|_{L^2(D_2)}
\leq C_{k,\varepsilon} \|f\|_{\gamma(D_2)} \|x^*\|.
\end{align*}
Now the result follows by taking the supremum over all $\|x^*\|\leq 1$.
\end{proof}

Let the strip $S_{\alpha}$ be given by $S_{\alpha} = \{z\in \C: |\text{Im}(z)|<\alpha\}$.

\begin{lemma}\label{lem:analyticgammadiscrte}
Let $f:S_{\alpha}\to X$ be holomorphic and let $0\leq a<b<\alpha$.
Let
\[\gamma(f) = \Big\|\sum_{n\in \Z} \gamma_n f(n)\Big\|_{L^2(\Omega;X)}.\]
Then  one has
\begin{align}
\label{eq:estgamma1}\sum_{j\in \{-1,1\}} \sup_{s\in [0,1]} & \Big(\gamma(f(\cdot+s+ ij a)) + \gamma(f'(\cdot+s+ ij a))\Big)
 \lesssim \|f\|_{\gamma(S_b;X)}
\\ & \label{eq:estgamma2} \leq \sum_{j\in \{-1,1\}} \|f(\cdot+ijb)\|_{\gamma(\R;X)}
\\ & \label{eq:estgamma3} \leq \sum_{j\in \{-1,1\}} \int_{0}^1 \gamma(f'(\cdot+s+ijb)) \, ds
\ \  + \ \ \gamma(f(\cdot+ ijb)),
\end{align}
where the convergence of each of the right-hand sides implies the converges of the previous term.
\end{lemma}
\begin{proof}
First assume $f\in \gamma(S_b;X)$ Fix a radius $0<r < \min\{1,b-a\}$, $s\in [0,1]$ and consider the disjoint balls $B_n = \{z:|z-n-s-ia|\leq r\}\subseteq S_{\alpha}$ for $n\in \Z$. Let $\phi_n = |B_n|^{-1/2} \one_{B_n}$ for $n\in \Z$. Then $(\phi_n)_{n\in \Z}$  is an orthonormal system in $L^2(S_b)$. By the mean value property (or via Cauchy's formula) one sees that
\[f(n+s+ia) = \frac{1}{|B_n|} \int_{B_n} f(z) \, |dz| = \frac{1}{r\pi^{\frac12}} I_f \phi_n.\]
It follows that for every $s\in [0,1]$ and $j\in \{-1,1\}$
\begin{align*}
r\pi^{\frac12} \gamma(f(\cdot+s+ ij a)) = \Big\|\sum_{n\in \Z} \gamma_n I_f \phi_n\Big\|_{L^2(\Omega;X)}\leq \|f\|_{\gamma(S_b;X)}.
\end{align*}
This proves the first estimate for $f$. For $f'$, we can use a similar argument. Consider the disjoint annuli $A_n = \{z:\frac{r}{2}<|z-n-s-ia|\leq r\}\subseteq S_{\alpha}$ for $n\in \Z$. Let $\psi_n = c \frac{\one_{A_n}}{(z-n-s-ia)}$
for $n\in \Z$ with
\[c^{-1} = \Big\|z\mapsto \frac{\one_{A_n}}{z-n-s-ia}\Big\|_{L^2(S_b)} = \int_{r/2}^r \int_{0}^{2\pi} |t e^{i x}|^{-2} t \, dx \, dt = 2\pi \log(2).\]
Then $(\psi_n)_{n\in \Z}$ is an orthonormal system in $L^2(S_b)$. Using Cauchy's formula one can check that
\[f'(n+s+ia) = \frac{1}{|A_n|} \int_{A_n} \frac{f(z)}{z-n-s-ia} \, |dz| = M  \cdot I_f \psi_n,\]
where $M = \frac{1}{c|A_n|} = \frac{4 \log(2)}{3\pi r^2}$.
It follows that for every $s\in [0,1]$ and $j\in \{-1,1\}$
\begin{align*}
M^{-1} \gamma(f'(\cdot+s+ ij a)) = \Big\|\sum_{n\in \Z} \gamma_n I_f \psi_n\Big\|_{L^2(\Omega;X)}\leq \|f\|_{\gamma(S_b;X)}.
\end{align*}
This completes the proof of \eqref{eq:estgamma1}.

To prove \eqref{eq:estgamma2} assume $f(\cdot+ijb)\in \gamma(\R;X)$ for $j\in \{-1,1\}$. We will use the Poisson formula for the strip (see \cite[1.10.3]{Tri} and \cite[Section 31]{RadH} for the Poisson formula for the strip $S_\alpha$ rotated by $90$ degrees):
\begin{equation}\label{eq:poissonstrip}
g(x+iy) = [k_y^{0} *g(\cdot+ib)](x) + [k_y^{1} *g(\cdot-ib)](x), \ \ \text{a.e.} \ x\in \R, |y|<\alpha
\end{equation}
if $g:S_\alpha\to \C$ is holomorphic on $S_{\alpha}$ and $L^2$-integrable on $\{z\in \C:\text{Im(z)}= \pm b\}$.
The kernels $k_y^j:\R\to \R$ are positive and satisfy
\[\|k_y^{0}\|_{L^1(\R)}  + \|k_y^{1}\|_{L^1(\R)} = \int_{\R} k_y^0(t)+k_y^1(t) \, dt = 1\]
for $-b<y<b$. As a consequence the mappings $K^{j}:L^2(\R)\to L^2(S_b)$ given by $K^{j}g (x+iy) = k_y^{j}*g(x)$ are bounded of norm $\leq 1$ and (see \cite[Proposition 4.4]{KalW04}) extend to $K^{j}:\gamma(L^2(\R);X)\to \gamma(L^2(S_b);X)$ of norm $\leq 1$ and moreover \eqref{eq:poissonstrip} holds with $g$ replaced by $f$. Therefore, we find
\begin{align*}
\|f\|_{\gamma(S_b;X)} &\leq  \|K^0 f(\cdot+ib)\|_{\gamma(S_b;X)} + \|K^1 f(\cdot-ib)\|_{\gamma(S_b;X)}
\\ & \leq \|f(\cdot+ib)\|_{\gamma(\R;X)} + \|f(\cdot-ib)\|_{\gamma(\R;X)}
\end{align*}
and the required estimate \eqref{eq:estgamma2} follows.

Finally, to prove \eqref{eq:estgamma3} we use the simple fact that for every $t\in \R$ we can write
\begin{align*}
f(t \pm i b)
= \sum_{n\in \Z} \one_{[n, n+1)}(t) f(n \pm ib) + \int_0^1 \sum_{n\in \Z} \one_{[n+s, n+1)}(t) f'(n+s\pm ib) \, ds.
\end{align*}
Now taking $\gamma$-norms with respect to $t\in \R$ on both sides we find
\begin{align*}
\|f(\cdot \pm b) \|_{\gamma(\R;X)}& \leq \Big\|\sum_{n\in \Z} \one_{[n, n+1)}(t) f(n \pm ib)\Big\|_{\gamma(\R,dt;X)} \\ & \ \ \ \ + \int_0^1 \Big\|\sum_{n\in \Z} \one_{[n+s, n+1)}(t) f'(n+s\pm ib) \Big\|_{\gamma(\R,dt;X)}\, ds
\\ & = \gamma(f(\cdot \pm  ib)) + \int_0^1 (1-s)^{\frac12} \gamma(f'(\cdot+s\pm ib)) \, ds
\end{align*}
from which the required result follows.
\end{proof}

\begin{lemma}\label{lem:estgammastandard}
Let $-\infty<a<b<\infty$.
\begin{enumerate}[$(1)$]
\item If $f\in W^{1,1}(a,b;X)$, then $f\in \gamma(a,b;X)$ and
\[\|f\|_{\gamma(a,b;X)} \leq (b-a)^{-\frac12}\|f(b)\| + (b-a)^{\frac12}\int_a^b \|f'(t)\|\, dt.\]
\item If $f\in \gamma(a,b;X)$ and $f'\in \gamma(a,b;X)$, then $f\in C([a,b];X))$ and
\[\sup_{t\in [a,b]}\|f(t)\|\leq (b-a)^{-\frac12} \|f\|_{\gamma(a,b;X)} + (b-a)^{\frac12} \|f'\|_{\gamma(a,b;X)}.\]
\end{enumerate}
\end{lemma}
\begin{proof}
The estimate (1) follows from \cite[Example 4.6]{KalW04} (also see \cite[Proposition 13.9]{Nee10}).

To prove (2) note that for $a\leq s<t\leq b$, $f(t) = f(s)+\int_{s}^t f'(r) \, dr = f(s) + I_{f'} \one_{[s,t]}$.
Multiplying by $\one_{(a,b)}(s)$ and integrating over $s$ we find that
\[f(t) (b-a) = I_f(\one_{[a,b]}) + \int_a^b  I_{f'} \one_{[s,t]}\, ds.\]
Therefore,
\[\|f(t)\| \leq (b-a)^{-\frac12} \|f\|_{\gamma(a,b;X)} + (b-a)^{\frac12} \|f'\|_{\gamma(a,b;X)}.\]
\end{proof}

Let $X_1$ and $X_2$ be vector spaces with norms $\|\cdot\|_{X_1}$ and $\|\cdot\|_{X_2}$ which embed in a Hausdorff topological vector space $V$. If we write $\|x\|_{X_1}\lesssim \|x\|_{X_2}$, this means that $x\in X_2$ implies $x\in X_1$ and the stated estimate holds true. This notation will be used below.

\begin{lemma}\label{lem:analyticYlemma}
Assume $f:S_{\alpha}\to X$ is holomorphic and $0\leq a<b<c<d<\alpha$. Let $Y= \gamma(\R;X)$ or $Y = L^p(\R;X)$ with $p\in [1, \infty]$.
Then for any integer $k\geq 1$,
\begin{align}
\sup_{s\in [-a,a]}\|f(\cdot+is)\|_{Y} & \lesssim \int_{-b}^b \|f(\cdot+is)\|_{Y} \, ds \label{eq:estgammacont1}
\\ & \lesssim \sum_{j=0}^k \|f^{(j)}(\cdot+is)\|_{\gamma((-b,b), ds; Y)}  \label{eq:estgammacont2}
\\ & \lesssim \sup_{s\in [-b,b]}\sum_{j=0}^{k+1} \|f^{(j)}(\cdot+is)\|_{Y}  \label{eq:estgammacont3}
\\ & \lesssim \sup_{s\in [-c,c]}\|f(\cdot+is)\|_{Y}\label{eq:estgammacont4}
\\ & \lesssim \sum_{j\in \{-1,1\}}\|f(\cdot+ijd)\|_{Y}. \label{eq:estgammacont5}
\end{align}
\end{lemma}

As a consequence of this result all the norms in Lemma \ref{lem:analyticgammadiscrte} are connected to the above expressions as well.

\begin{proof}
We first prove \eqref{eq:estgammacont4}. Assume $C:=\sup_{s\in [-c,c]}\|f(\cdot+is)\|_{Y}<\infty$ and let $R_{c} = \{x+iy: x\in [-1,1], y\in [-c, c] \}$. We can define $F:R_c\to Y$ by $F(x+iy)(t) = f(x+t+iy)$. We claim this function is holomorphic. To prove that $F$ is holomorphic it suffices by \cite[Theorem A.7]{ABHN} to show that it is bounded and $z\mapsto \lb F(z), g\rb$ is holomorphic for all $g\in G$, where $G\subseteq Y^*$ separates the points of $Y$. Indeed, $F$ is bounded since for each $x+iy\in R_{c}$, by translation invariance
\[\|F(x+iy)\|_{Y} = \|f(\cdot+x+iy)\|_{Y} = \|f(\cdot+iy)\|_{Y}\leq C.\]
Now let
\[G = \{\one_{I} \otimes x^*: \text{$I\subseteq \R$ is a bounded interval}, x^*\in X^*\}.\]
Then $G$ separates the points of $Y$. Moreover,
\[\lb F(x+iy), \one_{I} \otimes x^*\rb = \int_{I} \lb f(x+iy+t), x^*\rb \, dt \]
and the latter is holomorphic since it is the uniform limit of a sequence of holomorphic functions given by Riemann sums. Now the claim follows and moreover $\|F\|_{L^\infty(R_c;Y)}\leq C$.

From the claim and Lemma \ref{lem:analyticbddgamma} we find that for all integers $\ell\geq 0$, $F\in W^{\ell,\infty}(R_b;Y)$ and for all $x+iy\in R_b$,
\[\|f^{(j)}(x+iy)\|_Y= \|f^{(j)}(\cdot+x+iy)\|_Y = \|F^{(j)}(x+iy)\|_Y \lesssim \|F\|_{L^\infty(R_c;Y)} \leq C.\]
and \eqref{eq:estgammacont4} follows.

The estimates \eqref{eq:estgammacont2} and \eqref{eq:estgammacont3} are immediate from Lemma \ref{lem:estgammastandard}.

Finally we prove \eqref{eq:estgammacont1} and \eqref{eq:estgammacont5} by using a Poisson transformation argument.
As in the proof of Lemma \ref{lem:analyticgammadiscrte} one sees that for all $\theta\in (-b,b)\setminus [-a,a]$ and all $(t,s)\in \R\times [-a,a]$ we can write:
\begin{equation}\label{eq:poissonstrip2}
f(t+is) = [k_s^{0} *f(\cdot+i\theta)](t) + [k_s^{1} *f(\cdot-i\theta)](t),
\end{equation}
and hence
\[\|f(\cdot+is)\|_{Y} \leq \|f(\cdot+i\theta)\|_{Y} + \|f(\cdot-i\theta)\|_{Y}.\]
Now an integration over $\theta\in (a,b)$ gives \eqref{eq:estgammacont1}.

Estimate \eqref{eq:estgammacont5} can be proved in the same way if one takes $\theta = d$ and $s\in [-c,c]$.
\end{proof}

\section{Embedding results for holomorphic functions and type and cotype}

Recall that the strip $S_{\alpha}$ is given by $S_{\alpha} = \{z\in \C: |\text{Im}(z)|<\alpha\}$.
We also define the sector $\Sigma_{\sigma}$ by
\[\Sigma_{\sigma} = \{z\in \C\setminus\{0\}: |\text{arg}(z)|<\sigma\}.\]

\subsection{Type and cotype\label{subsec:type}}

In the next results we characterize the type and cotype of a Banach space $X$ by embedding results for holomorphic functions on a strip and sector, respectively. For more details on type and cotype we refer to \cite{DJT}.

Let $(\varepsilon_n)_{n\geq 1}$ be an i.i.d.\ sequence with $\P(\varepsilon_n = 1) = \P(\varepsilon_n = -1) = \frac12$. A space $X$ is said to have {\em type $p$} if there exists
a constant $\tau\ge 0$ such that for all $x_1,\dots,x_N$ in
$X$ we have
\begin{equation*}\label{eq:type}
\Big(\E \Big\| \sum_{n=1}^N \varepsilon_n x_n\Big\|^p\Big)^{1/p} \le \tau
\Big(\sum_{n=1}^N \| x_n\|^p\Big)^{1/p}.
\end{equation*}
A space $X$ is said to have {\em cotype $q$} if there exists
a constant $c\ge 0$ such that for all $x_1,\dots,x_N$ in $X$
we have
\begin{equation*}\label{eq:cotype}
\begin{aligned}
 \Big(\sum_{n=1}^N \|x_n\|^q\Big)^{1/q}
 & \le c\Big(\E \Big\| \sum_{n=1}^N \eps_n x_n\Big\|^q\Big)^{1/q},
\end{aligned}
\end{equation*}
with the obvious modification for $q=\infty$. Recall that
\begin{enumerate}
\item Every space $X$ has type $1$ and cotype $\infty$.
\item If $p_1>p_2$, then type $p_1$ implies type $p_2$.
\item If $q_1<q_2$, then cotype $q_1$ implies type $q_2$.
\item Hilbert spaces have type $2$ and cotype $2$.
\item $X = L^r$ for $1\leq r<\infty$ has type $r\wedge 2$ and cotype $r\vee 2$.
\end{enumerate}

\subsection{Statement of the main results}

In the next results we characterize type $p$ and cotype $q$ in terms of an embedding for holomorphic functions on both the strip and sector.
\begin{theorem}[Characterization of type]\label{thm:gammaanalytictype}
Let $X$ be a Banach space and $p\in [1, 2]$. Let $0\leq a<b<\alpha<\pi$. The following are equivalent:
\begin{enumerate}[$(1)$]
\item $X$ has type $p$
\item For all holomorphic functions $f:S_{\alpha}\to X$, the following estimate holds
\begin{align*}
\sum_{j\in \{-1,1\}} \|f(t + ija)\|_{\gamma(\R,dt;X)} \lesssim \sum_{j\in \{-1,1\}} \Big(\int_\R \|f(t+i j b)\|^p \, dt\Big)^{\frac1p},
\end{align*}
whenever the right-hand side is finite.
\item For all holomorphic functions $f:\Sigma_{\alpha}\to X$, the following estimates holds
\begin{align*}
\sum_{j\in \{-1,1\}} \|f(e^{ij a} t)\|_{\gamma(\R_+,\frac{dt}{t};X)}
\lesssim \sum_{j\in \{-1,1\}} \Big(\int_0^\infty \|f(e^{ij b}t)\|^p \, \frac{dt}{t}\Big)^{\frac1p}.
\end{align*}
whenever the right-hand side is finite.
\end{enumerate}
\end{theorem}

\begin{theorem}[Characterization of cotype]\label{thm:gammaanalyticcotype}
Let $X$ be a Banach space and $q\in [2, \infty]$. Let $0\leq a<b<\alpha<\pi$. The following are equivalent:
\begin{enumerate}[$(1)$]
\item $X$ has cotype $q\in [2, \infty]$.
\item For all holomorphic functions $f:S_{\alpha}\to X$, the following estimate holds
\begin{align*}
\sum_{j\in \{-1,1\}} \Big(\int_\R \|f(t + i ja)\|^q \, dt\Big)^{\frac1q} & \lesssim \sum_{j\in \{-1,1\}} \|f(t + ijb)\|_{\gamma(\R,dt;X)}
\end{align*}
whenever the right-hand side is finite.
\item For all holomorphic functions $f:\Sigma_{\alpha}\to X$, the following estimate holds
\begin{align*}
\sum_{j\in \{-1,1\}} \Big(\int_0^\infty \|f(e^{ij a}t)\|^q \, \frac{dt}{t}\Big)^{\frac1q} & \lesssim
\sum_{j\in \{-1,1\}} \|f(e^{ij b} t)\|_{\gamma(\R_+,\frac{dt}{t};X)}
\end{align*}
whenever the right-hand side is finite.
\end{enumerate}
\end{theorem}

\begin{remark}\label{rem:abouthm}
\
\begin{enumerate}[(i)]
\item In both results the condition $\alpha<\pi$ is not needed in the strip case. This is clear from the proofs
    below, but also follows by substituting $z\mapsto K z$ for a suitable $K>0$.
\item In case $X$ has type $2$ the embedding $L^2(T;X)\hookrightarrow \gamma(T;X)$ holds for any measure space $(T, \mathcal{T}, \mu)$. The reverse embedding in the cotype $2$ situation also holds (see \cite[Theorem 11.6]{Nee10}).
\item By Lemma \ref{lem:analyticgammadiscrte} for $\gamma$-norms and a similar version for $L^p$-norms, (2) in Theorem \ref{thm:gammaanalytictype} can be replaced by
\begin{align*}
\|f\|_{\gamma(S_a;X)} \lesssim \|f\|_{L^p(S_b;X)}.
\end{align*}
Similarly (2) of Theorem \ref{thm:gammaanalyticcotype} can be replaced this analogues norm estimate. See Remark \ref{rem:sector} for an alternative formulation of the norm estimate on the sector.
\item As a consequence of Theorems \ref{thm:gammaanalytictype} and \ref{thm:gammaanalyticcotype} and Kwapien's result (see \cite{Kwap72}), the $\gamma$-norm on a strip $S_a$ can be estimated from above and below by the $L^2$-norm on a slightly larger and smaller strip respectively, if and only if $X$ is isomorphic to a Hilbert space. The same holds for sectors.
\end{enumerate}
\end{remark}

\subsection{Proofs of the main results}

\begin{proof}[Proof of Theorem \ref{thm:gammaanalytictype} $(1)\Rightarrow (2)$]
Assume $X$ has type $p$. Let $a<b'<b$. By the type $p$ assumption, we have
$\gamma(f)\lesssim_{p,X} \Big(\sum_{n\in \Z} \|f(n)\|^p\Big)^{\frac1p}$. Therefore, by Lemma \ref{lem:analyticgammadiscrte} we find that for every $r\in [0,1]$,
\begin{align*}
\sum_{j\in \{-1,1\}}\|f(\cdot + ija)\|_{\gamma(\R_+;X)} & = \sum_{j\in \{-1,1\}}\|f(\cdot+r+ ija)\|_{\gamma(\R_+;X)} \\ & \lesssim \sum_{j\in \{-1,1\}} \int_{0}^1 \Big(\sum_{n\in \Z} \|f'(t+n+r+ijb')\|^p\Big)^{\frac1p}\, dt
\\ & \ \  + \sum_{j\in \{-1,1\}} \Big(\sum_{n\in\Z}\|f(r+n+ijb')\|^p\Big)^{\frac1p}.
\end{align*}
Taking $p$-th powers and integrating over all $r\in [0,1]$ we find
\begin{align*}
\sum_{j\in \{-1,1\}}\|f(\cdot + ija)\|_{\gamma(\R_+;X)} \lesssim \sum_{j\in \{-1,1\}} \|f'(\cdot+ijb')\|_Y + \|f(\cdot+ijb')\|_Y,
\end{align*}
where $Y = L^p(\R;X)$. By Lemma \ref{lem:analyticYlemma} the latter can be estimated by
\[\sum_{j\in \{-1,1\}} \Big(\|f'(\cdot+ijb')\|_Y + \|f(\cdot+ijb')\|_Y\Big) \lesssim \sum_{j\in \{-1,1\}} \|f(\cdot+ijb)\|_Y\]
which completes the proof of (1).
\end{proof}

\begin{proof}[Proof of Theorem \ref{thm:gammaanalyticcotype} $(1)\Rightarrow (2)$]
Assume $X$ has cotype $q$. Let $a<b$. Fix $t\in [0,1]$. First assume $q<\infty$. By the cotype  $q$ condition and Lemma \ref{lem:analyticgammadiscrte},
\begin{align*}
\sum_{j\in \{-1,1\}} \Big(\sum_{n\in\Z} \|f(t + i ja)\|^q \Big)^{\frac1q}
& \lesssim \sum_{j\in \{-1,1\}} \gamma(f(\cdot +t + i ja))
\\ & \lesssim  \sum_{j\in \{-1,1\}} \|f(\cdot+ijb)\|_{\gamma(\R;X)}.
\end{align*}
Taking $q$-th powers and integrating over all $t\in [0,1]$ yields the required result. If $q=\infty$ one should replace the $\ell^q$-sum on the left-hand side by a supremum over all $n$.
\end{proof}

\begin{proof}[Proof of $(2)\Leftrightarrow (3)$ for Theorems \ref{thm:gammaanalytictype} and \ref{thm:gammaanalyticcotype}]
The map $z\mapsto e^z$ is a holomorphic bijection from the strip $S_{\alpha}$ onto the sector $\Sigma_{\alpha}$. Note that for the substitution $s = e^{t}$ one has $dt=\frac{ds}{s}$ which gives the additional division by $t$ in (3) in both cases.
\end{proof}

\begin{remark}\label{rem:sector}
\
\begin{enumerate}[(i)]
\item Arguing as in the above proof of $(2)\Leftrightarrow (3)$, Remark \ref{rem:abouthm} can be used to see that (3) of Theorem \ref{thm:gammaanalytictype} can be replaced by
\begin{align*}
\|f\|_{\gamma(\Sigma_{a};X)} \lesssim \|f\|_{L^p(\Sigma_b;X)}
\end{align*}
and a similar assertion holds for (3) of Theorem \ref{thm:gammaanalyticcotype}.
\item By the same type of argument as in the proof of $(2)\Leftrightarrow (3)$ one can create discrete norms (dyadic) for holomorphic functions on sectors. Indeed, if $f$ is defined on a sector $\Sigma_{\alpha}$ we can apply Lemmas \ref{lem:analyticgammadiscrte} and \ref{lem:analyticYlemma} to $g$ given by $g(z) = f(2^z)$ on the strip $S_{\beta}$  with $\beta = \log(2)^{-1} \alpha$.
\end{enumerate}
\end{remark}

\begin{proof}[Proof of Theorem \ref{thm:gammaanalytictype} $(2)\Rightarrow (1)$]

{\em Step 1}:
Before we start the proof we introduce special functions $\phi_n$ and state some of their properties. Let $(c_n)_{n\geq 1}$  be a strictly  increasing sequence such that $c_n/(2\pi)\in \N$.
Let $\phi_n:\C\to \C$ be given by $\phi_n(z) = \text{sinc}(2\pi z- c_n)$ for $n\in \N$, where $\text{sinc}(z) = \frac{\sin(z)}{z}$. Then $\phi_n$ has the following properties:
\begin{itemize}
\item For every $t\in \R$,
\[|\phi_n(t+ijb)| = \frac{|\sin(2\pi t-c_n+ijb)|}{\big(|2\pi t-c_n|^2 + |b|^2\big)^{\frac12}}.\]
\item For every $t\in \R$,
\begin{align}\label{eq:phinest}
\frac{\cosh(b)\, |\sin(2\pi t-c_n)|}{\big(|t-c_n|^2 + |b|^2\big)^{\frac12}}\leq |\phi_n(t+ijb)| \leq \frac{2\cosh(b)}{\big(|2\pi t-c_n|^2 + |b|^2\big)^{\frac12}}.
\end{align}
\item For every $t\in \R$,
\[\phi_n( t+ijb) = \F(x\mapsto e^{i c_n x + jbx}\one_{(-1,1)})(t).\]
In particular, setting $b=0$ we see that the functions $(2^{-1/2}\phi_n)_{n\geq 1}$ form an orthonormal system in $L^2(-1,1)$.
\end{itemize}

{\em Step 2}: To derive type $p$ it suffices to consider $p\in (1, 2]$. Let $c_n = 3\pi 2^{n}$ for $n\in \N$.
To prove (1) note that by Lemma \ref{lem:analyticYlemma}, the assumption implies that for all $f$ as in (2),
\[\|f\|_{\gamma(\R;X)}\lesssim \sum_{j\in \{-1,1\}} \Big(\int_\R \|f(t+i j b)\|^p \, dt\Big)^{\frac1p}.\]
Fix $x_1, \ldots, x_N\in X$. Letting  $f(t) = \sum_{n=1}^N \phi_n(t) x_n$ by the orthogonality of the $\phi_n$'s  we find
\begin{align}\label{eq:gammafsum}
\|f\|_{\gamma(\R;X)} = 2^{1/2}\Big\|\sum_{n=1}^N \gamma_n x_n\Big\|_{L^2(\Omega;X)}.
\end{align}
It now suffices to show that $\|f(\cdot+i j b)\|^p_{L^p(\R;X)}\lesssim \sum_{n=1}^N \|x_n\|^p$.
In the proof we use the following estimate for scalars $t:=(t_n)_{n=1}^N$ and $s:=(s_n)_{n=1}^N$:
\begin{equation}\label{eq:Holderextended}
\sum_{n=1}^\infty |s_n t_n| \leq \|s\|_{\ell^p} \|t\|_{\ell^{p'}} \leq \|s\|_{\ell^p} \|t\|_{\ell^{p}},
\end{equation}
where the last estimate holds as $p'\geq p$.

Let $a_n = 2^{n}$, so that $c_n$ is the midpoint of $[2\pi a_n, 2\pi a_{n+1}]$ for $1\leq n\leq N$. We now split the $L^p$-norm of $f(\cdot+ijb)$ as
\begin{align*}
\|f(\cdot+i j b)\|^p_{L^p(-\infty, a_1;X)} + \|f(\cdot+i j b)\|^p_{L^p(a_{N+1}, \infty;X)}+ \sum_{m=1}^{N} \|f(\cdot+i j b)\|^p_{L^p(a_m,a_{m+1};X)}.
\end{align*}
To finish the proof we estimate each of the three terms separately. Since $\|x+y\|^p\leq 2^{p-1} (\|x\|^p + \|y\|^p)$, for $1\leq m\leq N$ we can write
\begin{align*}
&\|f(\cdot+i j b)\|^p_{L^p(a_m,a_{m+1};X)}
\\ & \lesssim \int_{a_m}^{a_{m+1}} \|x_m\|^p |\phi_m(t+ ijb)|^p + \Big(\sum_{n\neq m} \|x_n\| |\phi_n(t+ ijb)| \Big)^p \, dt
\\ & \stackrel{(i)}{\lesssim} \|x_m\|^p  + \int_{a_m}^{a_{m+1}} \sum_{n\neq m} \|x_n\|^p   \cdot \sum_{n\neq m} |\phi_n(t+ ijb)|^{p}  \, dt
\\ & =  \|x_m\|^p  + \sum_{n\neq m} \|x_n\|^p  \sum_{n\neq m} \int_{a_m}^{a_{m+1}} |\phi_n(t+ ijb)|^{p}  \, dt,
\end{align*}
where we used $\sup_{m}\int_{\R} |\phi_m(t+ijb)|^p\, dt <\infty$ and \eqref{eq:Holderextended} in (i). For the terms $1\leq n\leq m-1$ we can use \eqref{eq:phinest} in order to get
\begin{equation}\label{eq:helpterm1}
\begin{aligned}
\sum_{n=1}^{m-1} \int_{a_m}^{a_{m+1}} |\phi_n(t+ ijb)|^{p}  \, dt & \lesssim \frac{(m-1)2^m }{\big(|2\pi a_m -c_{m-1}|^2 + |b|^2\big)^{\frac{p}{2}}}
\\ & \lesssim \frac{(m-1) 2^m }{|2^{m+1} - 3 \cdot 2^{m-1}|^p} \lesssim (m-1) 2^{-m(p-1)}.
\end{aligned}
\end{equation}
For the terms with $n\geq m+1$, we have
\begin{equation}\label{eq:helpterm2}
\begin{aligned}
\sum_{n=m+1}^{N} \int_{a_m}^{a_{m+1}} |\phi_n(t+ ijb)|^{p}  \, dt & \lesssim
\sum_{n=m+1}^{N} \frac{2^m }{\big(|c_{n} - 2\pi a_{m+1}|^2 + |b|^2\big)^{\frac{p}{2}}}
\\ & \lesssim \sum_{n=m+1}^\infty \frac{2^{-m(p-1)} }{|3\cdot 2^{n-m-1} - 2|^p}\lesssim 2^{-m(p-1)}.
\end{aligned}
\end{equation}
As the right-handsides of \eqref{eq:helpterm1} and \eqref{eq:helpterm2} are summable in $m$ we can conclude
\[\sum_{m=1}^N \|f(\cdot+i j b)\|^p_{L^p(a_m,a_{m+1};X)} \lesssim \sum_{m=1}^N \Big(\|x_m\|^p  + 2^{-m(p-1)} \sum_{n\neq m} \|x_n\|^p\Big) \lesssim  \sum_{m=1}^N \|x_m\|^p.\]

Next we estimate the $L^p$-norm on $(-\infty, a_1)$. In order to do so note that again by \eqref{eq:Holderextended} and \eqref{eq:phinest}
\begin{align*}
\|f(\cdot+i j b)\|^p_{L^p(-\infty, a_1;X)} & \leq \int_{-\infty}^{a_1} \sum_{n=1}^N \|x_n\|^p  \cdot \sum_{n=1}^N |\phi_n(t+ ijb)|^{p}  \, dt
\\ & \leq \sum_{n=1}^N \|x_n\|^p  \sum_{n=1}^N   \int_{-\infty}^{a_1} |\phi_n(t+ ijb)|^{p}  \, dt
\\ & \lesssim \sum_{n=1}^N \|x_n\|^p  \sum_{n=1}^N   \int_{-\infty}^{a_1} \frac{1}{|c_n- 2\pi t|^p} \, dt
\\ & \lesssim \sum_{n=1}^N \|x_n\|^p  \sum_{n=1}^N   \frac{1}{|3\cdot 2^n - 2|^{p-1}}
\lesssim \sum_{n=1}^N \|x_n\|^p
\end{align*}
since $p>1$. The estimate for the $L^p$-norm on $(a_{N+1}, \infty)$ is proved similarly.
\end{proof}

\begin{proof}[Proof of Theorem \ref{thm:gammaanalyticcotype} $(2)\Rightarrow (1)$]
Let $\phi_n$ be as in Step 1 of the previous proof, but this time with $c_n = 3\pi r^{n}$, where $r>2$ is an even integer which is fixed for the moment. By Lemma \ref{lem:analyticYlemma}, the assumption implies that for all $f$ as in (2)
\[ \Big(\int_\R \|f(t)\|^q \, dt\Big)^{\frac1q} \lesssim \sum_{j\in \{-1,1\}} \|f(t + ijb)\|_{\gamma(\R,dt;X)}.\]
Let $f(t) = \sum_{n=1}^N \phi_n(t) x_n$.
As the Fourier transform is an isometry on $\gamma(\R;X)$ and $\hat{f}$ has support on $[-1,1]$, the ideal property yields
\begin{align*}
\|f(\cdot+ijb)\|_{\gamma(\R;X)} & = \|\F f(\cdot+ijb)\|_{\gamma(\R;X)} = \|e^{jb x} \hat{f}(x)\|_{\gamma(\R;X)}
\\ & \leq e^{b} \|\hat{f}\|_{\gamma(-1,1;X)} = 2^{1/2} e^b \Big\|\sum_{n=1}^N \gamma_n x_n\Big\|_{L^2(\Omega;X)},
\end{align*}
where we used that $(2^{-1/2}\hat{\phi}_n)_{n=1}^N$ is an orthonormal system. Let $a_n = r^{n}$, so that $c_n$ is in the interval $[2\pi a_n, 4\pi a_{n}]$ for $1\leq n\leq N$. Clearly, we may estimate
\begin{align*}
\|f\|^q_{L^q(\R;X)} & \geq \sum_{m=1}^{N-1} \int_{a_m}^{2a_{m}} \|f(t)\|^q\, dt.
\end{align*}
Now since $\|x+y\|^q\geq \big|\|x\|-\|y\|\big|^q\geq 2^{1-q}\|x\|^q - \|y\|^q$, we find
\begin{align*}
\int_{a_m}^{2a_{m}} \|f(t)\|^q\, dt & \geq \int_{a_m}^{2a_{m}} 2^{1-q} \|x_m\|^q |\phi_m(t)|^q - \Big\|\sum_{n\neq m} x_n \phi_n(t)\Big\|^q\, dt.
\end{align*}
Note that to obtain a lower estimate of the latter, we have to be a careful with the minus sign. Clearly,
\begin{align*}
\int_{a_m}^{2a_{m}} 2^{1-q} \|x_m\|^q |\phi_m(t)|^q \geq 2^{1-q} \|x_m\|^q  \int_{-1}^{1} \frac{|\sin(2\pi t)|^q}{|2\pi t|^q}  \, dt \geq C_q \|x_m\|^q,
\end{align*}
and the latter is not depending on $r$.
By H\"older's inequality
\begin{align*}
\int_{a_m}^{2a_{m}}  \Big(\sum_{n\neq m} \|x_n\| |\phi_n(t)|\Big)^q\, dt& \leq \int_{a_m}^{2a_{m}}  \sum_{n\neq m} \|x_n\|^q \Big(\sum_{n\neq m} |\phi_n(t)|^{q'}\Big)^{q-1}\, dt
\\ & \leq \sum_{n\neq m} \|x_n\|^q  \int_{a_m}^{2a_{m}}  \Big(\sum_{n\neq m} |\phi_n(t)|^{q'}\Big)^{q-1}\, dt.
\end{align*}
Using convexity the latter integral can be estimated by
\begin{align*}
2^{q-2} \int_{a_m}^{2a_{m}}  \Big(\sum_{n=1}^{m-1} |\phi_n(t)|^{q'}\Big)^{q-1}\, dt + 2^{q-2} \int_{a_m}^{2a_{m}}  \Big(\sum_{n> m} |\phi_n(t)|^{q'}\Big)^{q-1}\, dt.
\end{align*}
The first term we estimate by using \eqref{eq:phinest} and the worst case times $(m-1)^{q-1}$:
\begin{align*}
\int_{a_m}^{2a_{m}}  \Big(\sum_{n=1}^{m-1} |\phi_n(t)|^{q'}\Big)^{q-1}\, dt
&  \leq C_{q,b}\int_{a_m}^{2a_{m}}  (m-1)^{q-1}  \frac{1}{|2\pi t-c_{m-1}|^q} \, dt
\\ & \leq C_{q,b} (m-1)^{q-1} \frac{r^{m}}{|\pi(2 r^m - 3 r^{m-1})|^{q} } .
\\ & = \frac{C_{q,b}}{\pi^q} (m-1)^{q-1} r^{-m(q-1)}\frac{1}{|(2  - 3/r)|^{q} }=:T_{1,m}(r),
\end{align*}
where $C_{q,b} = 2^q \cosh^q(b)$. The second term can be estimated as
\begin{align*}
\int_{a_m}^{2a_{m}}  \Big(\sum_{n> m} |\phi_n(t)|^{q'}\Big)^{q-1}\, dt
&\leq C_{q,b}\int_{a_m}^{2a_{m}}  \Big(\sum_{k=1}^{\infty}  \frac{1}{|c_{m+k} - 2\pi t|^{q'}}\Big)^{q-1}\, dt
\\ &\leq C_{q,b} r^m r^{-m q }  \Big(\sum_{k=1}^{\infty}  \frac{1}{|2\pi r^k - 4\pi |^{q'}}\Big)^{q-1}\, dt
=: T_{2,m}(r).
\end{align*}
Therefore, summing over $1\leq m\leq N$ yields
\[\|f\|^q_{L^q(\R;X)} \geq \Big(C_q - 2^{q-2} \sum_{m\geq 1} T_{1, m}(r) + T_{2,m}(r)\Big) \sum_{m=1}^N \|x_m\|^q.\]
The latter series converges and moreover by the dominated convergence theorem $\sum_{m\geq 1} T_{1, m}(r) + T_{2,m}(r) \to 0$ as $r\to \infty$. Now by letting $r\to \infty$ and combining all estimates it follows that
\[\Big\|\sum_{n=1}^N \gamma_n x_n\Big\|_{L^2(\Omega;X)} \leq C \Big(\sum_{m=1}^N \|x_m\|^q)^{\frac1q}.\]
\end{proof}

\begin{remark}
The proofs in the previous sections are all based on the mean value property. Therefore, it should be possible to extend all results to the setting of harmonic functions $f:D\to X$ where $D\subset \R^d$. For the applications we have in mind, one needs estimates for holomorphic functions $f$, and therefore we only consider this setting.
\end{remark}

\subsection{An alternative method\label{sec:embeddinggeneral}}

In Theorems \ref{thm:gammaanalytictype} and \ref{thm:gammaanalyticcotype} we have seen an embedding result for $\gamma$-norms and $L^p$-norms of holomorphic functions under the assumptions that $X$ has cotype $q$. Embeddings of this type for functions which are not necessarily holomorphic but in suitable Sobolev spaces, have been obtained in \cite{KNVW}. Below we give a simple proof of $W^{1,p}(\R;X)\hookrightarrow \gamma(\R;X)$ and $\gamma^1(\R;X)\hookrightarrow L^q(\R)$ under the assumption that $X$ has type $p$ and cotype $q$. Combining the above result with Lemma \ref{lem:analyticYlemma}, yields another proof of the implications (1)$\Rightarrow$ (2) in Theorems \ref{thm:gammaanalytictype} and \ref{thm:gammaanalyticcotype}.

\begin{proposition}\label{prop:embeddingtypegamma}
Let $X$ be a Banach space, $p\in [1, 2]$ and $q\in [2, \infty]$. The following assertions hold:
\begin{enumerate}
\item The space $X$ has type $p$ if and only if
there is a constant $C$ such that for every $u\in W^{1,p}(I;X)$,
\[\|u\|_{\gamma(\R;X)} \leq C \|u\|_{L^p(\R;X)} + \|u'\|_{L^p(\R;X)}.\]
\item The space $X$ has with cotype $q$ if and only if
there is a constant $C$ such that for every $u\in \gamma^1(I;X)$,
\[\|u\|_{L^q(\R;X)} \leq C \|u\|_{\gamma(\R;X)} + \|u'\|_{\gamma(\R;X)}.\]
\end{enumerate}
\end{proposition}
The main ingredient in the proof is a simple randomization lemma taken from \cite[Lemma 2.2]{KNVW}.
\begin{lemma}\label{lem:disjointtype}
Let $X$ be a Banach space with type $p$ and cotype $q\in [2, \infty)$. Let $(S, \Sigma, \mu)$ be a measure space.
\begin{enumerate}[$(1)$]
\item Let $(S_n)_{n\geq 1}$ in $\Sigma$ be a sequence of disjoint sets. Then for every $f\in \g(S;X)$,
\[\Big(\sum_{n\geq 1} \|f\|_{\gamma(S_n;X)}^q\Big)^{\frac1q} \leq  C_{q,X} \|f\|_{\gamma(S;X)}.\]
\item Let $(S_n)_{n\geq 1}$ in $\Sigma$ be a sequence of disjoint sets such that $\bigcup_{n\geq 1} S_n = S$. If $f\in \gamma(S_n;X)$ for each $n\geq 1$ and $\Big(\sum_{n\geq 1}\|f\|_{\gamma(S_n;X)}^p\Big)^{\frac1p}<\infty$, then $f\in \gamma(S;X)$ and \[ \|f\|_{\gamma(S;X)}\leq C_{p,X}\Big(\sum_{n\geq 1} \|f\|_{\gamma(S_n;X)}^p\Big)^{\frac1p}<\infty.\]
\end{enumerate}
\end{lemma}

\begin{proof}[Proof of Proposition \ref{prop:embeddingtypegamma}]

(2): \
Assume $X$ has cotype $q\in [2, \infty)$. Let $u\in \gamma^1(\R_+;X)$. By writing $u$ as the integral of its derivative one sees that $u$ has a continuous version (see \cite{NVWgamma}).
Let $\varphi:\R\to \R$ be given by $\varphi(x) = 1-x$ for $x\in [0,1]$ and $\varphi(x) = 0$ for $x>1$.
For each $t\in \R$, we can write
\begin{align}\label{eq:urepresen}
u(t) = \int_{t}^{t+1} \varphi(x-t) u'(x) \, dx + \int_{t}^{t+1} \varphi'(x-t) u(x) \, dx.
\end{align}
Since $\|\varphi\|_{L^2(0,1)} \leq 1$ and $\|\varphi'\|_{L^2(0,1)} \leq 1$, we find
\[\|u(t)\|\leq \|u'\|_{\gamma(t, t+1;X)} + \|u\|_{\gamma(t, t+1;X)}.\]
Taking $q$-th powers on both sides and integrating over $t\in \R$, we find that
\begin{align*}
\|u\|_{L^q(\R;X)} \leq \big\|t\mapsto \|u'\|_{\gamma(t, t+1;X)}\big\|_{L^q(\R;X)} + \big\|t\mapsto \|u\|_{\gamma(t, t+1;X)}\big\|_{L^q(\R;X)}.
\end{align*}
By Lemma \ref{lem:disjointtype} with $I_n = (2n, 2n+2)$ and $J_n = (2n+1, 2n+3)$ for $n\in \Z$,
\begin{align*}
\big\|t\mapsto \|u\|_{\gamma(t, t+1;X)}\big\|_{L^q(\R;X)}
& \leq \Big(\sum_{j\in \Z} \int_{j}^{j+1} \|u\|_{\gamma(t, t+1;X)}^q \, dt\Big)^{\frac1q}
\\ & \leq \Big(\sum_{j=0}^\infty \|u\|_{\gamma(j, j+2;X)}^q\Big)^{\frac1q}
\\ & \leq \Big(\sum_{n\in \Z}^\infty \|u\|_{\gamma(I_n;X)}^q\Big)^{\frac1q} + \Big(\sum_{n=0}^\infty \|u\|_{\gamma(J_n;X)}^q\Big)^{\frac1q}
\\ & \leq 2 C_{q,x} \|u\|_{\gamma(\R;X)}^q.
\end{align*}
Since the term $u'$ can be estimated in exactly the same way we find
\begin{align}\label{eq:uLqnormgamma}
\|u\|_{L^q(\R_+;X)}  \leq 4 C_{q,x} \|u\|_{\gamma^1(\R_+;X)}.
\end{align}

To prove the converse it suffices to apply $\|u\|_{L^q(\R_+;X)}\leq C \|u\|_{\gamma^{1}(\R_+;X)}$ to the function
$u = \sum_{j=1}^n \varphi_j x_j$, where $\varphi:\R\to \R$ is given by $\varphi(x) = 1-2|x|$ for $x\in [-1/2, 1/2]$ and zero otherwise and $\varphi_j(t)  = \varphi(t-j)$

(1): \  First assume $X$ has type $p$. Taking $\gamma$-norms for $t\in (j, j+1)$ in \eqref{eq:urepresen} we find
\begin{align*}
\|u\|_{\gamma(j, j+1;X)} & \leq \int_{j}^{j+2} \|\varphi\|_{L^2(t, t+1)} \|u'(x)\| + \|\varphi'\|_{L^2(t, t+1)} \|u(x)\| \, dx
\\ & \lesssim \|u\|_{W^{1,1}(j, j+1;X)} + \|u\|_{W^{1,1}(j+1, j+2;X)}.
\end{align*}
Taking $p$-th powers and summing over all $j\in \Z$, and applying Lemma \ref{lem:disjointtype} we find
\begin{align}\label{eq:ugammaestWW}
\|u\|_{\gamma(\R;X)} & \leq C_{p,X} \Big(\sum_{j\in \Z} \|u\|_{W^{1,1}(j, j+1;X)}^p\Big)^{\frac1p}\leq \|u\|_{W^{1,p}(\R;X)}.
\end{align}
To see that the embedding implies type $p$, one can take $u = \sum_{j=1}^n \varphi_j x_j$, where $\varphi_j$ is as before.
\end{proof}

\section{Applications}

In this section we present two applications of Theorems \ref{thm:gammaanalytictype} and \ref{thm:gammaanalyticcotype}.
Before doing so we briefly introduce the so-called $H^\infty$-calculus of a sectorial operator $A$.

\subsection{Preliminaries on $H^\infty$-calculus}

For details on the $H^\infty$-calculus we refer the reader to \cite{Haase:2,KunWeis04}. We repeat the part of the theory which we need below.

Let $\sigma\in (0,\pi)$. A linear operator $(A, D(A))$ on $X$ is called {\em sectorial of type $\sigma$} if $D(A)$ is dense in $X$, $A$ is injective and has dense range, $\sigma(A)\subseteq \overline{\Sigma_\sigma}$,
and for all $\sigma'\in(\sigma,\pi)$ the set
$$ \big\{z(z-A)^{-1}: \ z\in \C\setminus\{0\}, \  |\arg(z)|> \sigma'\big\}$$
is uniformly bounded.

Recall that $H^\infty(\Sigma_{\sigma})$ stands for the space of bounded holomorphic functions on $\Sigma_{\sigma}$. Examples of operators with an $H^\infty$-calculus include most differential operators on $L^p$-spaces with $p\in (1, \infty)$.
Below the $H^\infty$-calculus of $A$ will be applied through square function estimates. Recall that $v\in H^\infty_0(\Sigma_{\sigma})$ if it is in $H^\infty(\Sigma_{\sigma})$ and there exists an $\varepsilon>0$ such that $|v(z)|\leq \frac{|z|^{\varepsilon}}{1+|z|^{2\varepsilon}}$ for $z\in \Sigma_{\sigma}$. For a sectorial operator $A$ of angle $<\sigma$ and $v\in H^\infty_0(\Sigma_{\sigma})$, we can define $v(tA)$ using the Dunford calculus. The operator $A$ is said to have a {\em bounded $H^\infty$-calculus of angle $\sigma$} if $\|v(A)x\|\leq C\|v\|_{\infty} \|x\|$ for all such $v$, where $C$ is independent of $v$.

Assume that $A$ has a bounded $H^\infty$-calculus of angle $\sigma$. Assume $v\in H^\infty_0(\Sigma_{\sigma})$ is nonzero. Then
\begin{equation}\label{eq:lowerHinfty}
\|x\|\leq C\|t\mapsto v(tA) x\|_{\gamma(\R_+,\tfrac{dt}{t};X)},
\end{equation}
where $x\in X$ is such that the right-handside is finite. Moreover,
\begin{equation}\label{eq:upperHinfty}
\|t\mapsto v(tA) x\|_{\gamma(\R_+,\tfrac{dt}{t};X)}\leq  C\|x\|,  \ \ x\in X
\end{equation}
provided $X$ has finite cotype. Moreover, if the above estimates hold for some nonzero $v\in H^\infty_0(\Sigma_{\sigma})$ with $A$ replaced by $e^{\pm \phi i}A$ for some $\phi>0$ and all $x$ in a dense subspace of $X$, then one can also deduce that $A$ has a bounded $H^\infty$-calculus. Proofs of these results can be found in \cite{HaHa} and \cite[Section 7]{KalW04}.

\subsection{Littlewood-Paley-Stein $g$-functions}

In \cite{Stein:topics} continuous Littlewood--Paley estimates have been introduced which now are usually referred to as Littlewood-Paley-Stein $g$-functions. In \cite{Xu98} one-sided $L^p$-estimates for these $g$-functions are studied in a Banach space setting for the Poisson semigroup and in \cite{MTX06} for more general diffusion semigroups. It turns out that such $L^p$-estimates are equivalent to martingale type and cotype of the underlying Banach space. For similar results in the Laguerre setting we refer to \cite{BFRMST11}.

In \cite{Hyt07}, \cite{KaiWei08} and \cite{KalW04} the continuous square functions from \cite{Stein:topics} are generalized to the Banach space valued situation in different ways using $\gamma$-norms (also see \cite{BCCFRM12,betancor2012umd}).

Below we will combine Theorems \ref{thm:gammaanalytictype} and \ref{thm:gammaanalyticcotype} and \eqref{eq:lowerHinfty} and \eqref{eq:upperHinfty} to obtain $L^p$-estimates for certain classes of diffusion operators. This leads to a different approach to the estimates obtained in \cite{MTX06} and it is applicable to a wider class of operators.

We start with the following general Littlewood-Paley-Stein inequality.
\begin{proposition}\label{prop:generalPL}
Assume $X$ has type $p\in [1, 2]$ and cotype $q\in [2, \infty]$ and assume that $A$ is sectorial and has a bounded $H^\infty$-calculus of angle $<\phi$. Fix a nonzero $f\in H^{\infty}_0(\Sigma_{\phi})$. Then for all $x\in X$,
\begin{equation}\label{eq:LpLqest}
\|t\mapsto f(tA)x\|_{L^q(\R_+,\frac{dt}{t};X)} \lesssim \|x\|\lesssim \|t\mapsto f(tA)x\|_{L^p(\R_+,\frac{dt}{t};X)},
\end{equation}
where the second estimate holds whenever the right-hand  side is finite.
\end{proposition}

\begin{proof}
If $q=\infty$, then it suffices to note that $\|f(tA)x\|\leq C \|x\|$. Next assume $q<\infty$.
By \eqref{eq:upperHinfty}  we find
$\|t\mapsto f(tA)x\|_{\gamma(\R_+,\frac{dt}{t};X)}\leq \|x\|$. Thus the first estimate in \eqref{eq:LpLqest} follows from Theorem \ref{thm:gammaanalyticcotype} (3) with $a=0$.

For the second estimate in \eqref{eq:LpLqest} we use a duality argument. Define $g \in H^\infty_0(\Sigma_{\sigma})$ by $g(z) = c \overline{f(\overline{z})}$ with $c = \Big(\int_0^\infty |f(t)|^2 \, \frac{dt}{t}\Big)^{-1}$. Then
\[\int_0^\infty f(t) g(t) \, \frac{dt}{t} = 1.\]
Choose $x\in X$ such that the right-hand side of \eqref{eq:LpLqest} is finite  and let $x^*\in R(A^*)\cap D(A^*)$.
Note that $A^*$ has a bounded  $H^\infty$-calculus on $\overline{R(A^*)\cap D(A^*)}$ (see \cite[Appendix A]{KunWeis04}).
Now by an approximation argument (see \cite[Theorem 5.2.6]{Haase:2})) one can show that
\[\lb x, x^*\rb = \int_0^\infty \lb f(tA)x, g(tA)^*x^* \rb \, \frac{dt}{t}.\]
By H\"older's inequality and the fact that $X^*$ has cotype $p'$ (see \cite[Proposition 11.10]{DJT}) we find
\begin{align*}
|\lb x, x^*\rb|& \leq \|t\mapsto f(tA)x\|_{L^p(\R_+,\frac{dt}{t};X)} \|t\mapsto g(tA^*)x^*\|_{L^{p'}(\R_+,\frac{dt}{t};X^*)}.
\\ & \lesssim \|t\mapsto g(tA)x\|_{L^p(\R_+,\frac{dt}{t};X)} \|x^*\|,
\end{align*}
where in the last step we applied the result we have already proved in the cotype case. The required estimate now follows since $R(A^*)\cap D(A^*)$ is dense in $X^*$.
\end{proof}

Let $r\in (1, \infty)$ and let $(\Omega, \Sigma, \mu)$ be a $\sigma$-finite measure space. An operator $A$ on $L^r(\Omega)$ will be called a {\em diffusion operator} if it is a sectorial operator of angle $<\pi/2$ and it satisfies the following properties
\begin{enumerate}
\item For every $t\geq 0$ and $x\in L^r(\Omega)$ with $f\geq 0$, $e^{-t A} x\geq 0$.
\item For all $t\geq 0$, $\|e^{-tA}\|\leq 1$.
\end{enumerate}
Now fix a Banach space $X$. By positivity each of the operators $e^{-tA}$ has a tensor extension to a bounded operator $\textbf{T}(t)$ on $L^r(\Omega;X)$ (see \cite[Theorem V.1.12]{GarRub}), and in this way $\textbf{T}$ is a strongly continuous semigroup on $L^r(\Omega;X)$. Let $\textbf{A}$ be the generator of $\textbf{T}$.

For details on UMD spaces we refer to \cite{Burk01, HNVW1, Rubio86}.
\begin{theorem}\label{thm:estimatingLrLqLP}
Let $X$ be a UMD space with type $p\in [1, 2]$ and cotype $q\in [2, \infty]$.
Let $A$ be a diffusion operator on $L^r(\Omega)$ with $r\in (1, \infty)$ fixed and let $\textbf{A}$ be its extension to $L^r(\Omega;X)$ as above.
Choose a nonzero $f\in H^{\infty}_0(\Sigma_{\phi})$ with $\phi>\pi/2$. Then for all $x\in L^r(\Omega;X)$ the following estimates hold
\begin{align*}
\|t\mapsto f(t\textbf{A})x\|_{L^r(\Omega;L^q(\R_+,\frac{dt}{t};X))} & \lesssim \|x\|_{L^r(\Omega;X)}\lesssim \|t\mapsto f(t\textbf{A})x\|_{L^r(\Omega;L^p(\R_+,\frac{dt}{t}))}.
\end{align*}
\end{theorem}

\begin{proof}
By the first part of the proof of \cite[Theorem 9.3]{KalW04}, $\textbf{A}$ has a bounded $H^\infty$-calculus for every angle $\sigma>\pi/2$.
Now fix $x\in L^r(\Omega;X)$. Then the function $z\mapsto  f(zA)x$ is an holomorphic function on a sector $\Sigma_{\varepsilon}$ and by \cite[Proposition 2.6]{NVW07a} and \eqref{eq:upperHinfty},
\[\|t\mapsto f(t\textbf{A}) x\|_{L^r(\Omega;\gamma(\R_+,\frac{dt}{t};X))} \eqsim_r \|t\mapsto f(t\textbf{A})x\|_{\gamma(\R_+,\frac{dt}{t};L^r(\Omega;X))}\lesssim  \|x\|_{L^r(\Omega;X)}.\] By \cite[Theorem 1.1]{DHanalytic} there exists a strongly measurable function $\zeta:\Omega\times \Sigma_{\varepsilon}\to X$ such that $\zeta(\cdot, z) = f(z\textbf{A})x$ almost everywhere and for every $\omega\in \O$, $z\mapsto \zeta(\omega, z)$ is holomorphic. Hence applying Theorem \ref{thm:gammaanalyticcotype} pointwise in $\omega\in \O$ we find that
\begin{align*}
\|t\mapsto f(t\textbf{A})x\|_{L^r(\Omega;L^q(\R_+,\frac{dt}{t};X))} & = \|\zeta\|_{L^r(\Omega;L^q(\R_+,\frac{dt}{t};X))}
\lesssim  \|\zeta\|_{L^r(\Omega;\gamma(\R_+,\frac{dt}{t};X))} \\ & = \|t\mapsto f(t\textbf{A})x\|_{L^r(\Omega;\gamma(\R_+,\frac{dt}{t};X))}\lesssim \|x\|_{L^r(\Omega;X)}.
\end{align*}

For the other estimate one can argue similarly as in the proof of Proposition \ref{prop:generalPL}. Indeed, note that every UMD space is reflexive and thus has the Radon-Nikodym property and the dual of $L^r(\Omega;L^{p}(\R_+,\frac{dt}{t};X))$ is $L^{r'}(\Omega;L^{q'}(\R_+,\frac{dt}{t};X^*))$ (see \cite[Theorem IV.1]{DieUhl} or \cite{HNVW1}). Moreover, $X^*$ has cotype $p'$ and we can consider $\textbf{A}^*$ on $L^{r'}(\Omega;X^*)$. As before $\textbf{A}^*$ has an $H^\infty$-calculus when restricted to the closure of $D(\textbf{A}^*)\cap R(\textbf{A}^*)$.
\end{proof}

\begin{remark}
Assume the conditions of Theorem \ref{thm:estimatingLrLqLP} are satisfied.
One could apply Proposition \ref{prop:generalPL} directly to $\textbf{A}$ defined on the space $L^r(\Omega;X)$. As this space has type $r\wedge p$ and cotype $r\vee q$, this would yield
\[\|t\mapsto f(tA)x\|_{L^{q\vee r}(\R_+,\frac{dt}{t};L^r(\Omega;X))} \lesssim \|x\|_{L^r(\Omega;X)}\lesssim \|t\mapsto f(tA)x\|_{L^{p\wedge r}(\R_+,\frac{dt}{t};L^r(\Omega;X))}.\]
However, by Minkowski's inequality one sees that this estimate is a consequence of Theorem \ref{thm:estimatingLrLqLP} as well.
\end{remark}

As an immediate consequence we obtain a Littlewood-Paley-Stein estimate for the subordinated semigroups of $\textbf{A}$.
Recall that $\textbf{A}$ is sectorial of any angle $>\pi/2$. Therefore, for $\alpha\in (0,1)$, $\textbf{A}^{\alpha}$ is sectorial of any angle $<\pi\alpha/2$. Let $\textbf{T}_{\alpha}(t) = e^{-t \textbf{A}^{\alpha}}$ for $t\geq 0$. Note that for $\alpha =  \frac12$ and $t\geq 0$,
\[\textbf{T}_{\frac12}(t)x = \frac{1}{\pi} \int_0^\infty \frac{e^{-u}}{\sqrt{u}} \textbf{T}(\tfrac{t}{4u})x \, dt\]
is the abstract Poisson semigroup (see \cite{MTX06} and \cite[Example 3.4.6]{Haase:2}) associated to $\textbf{A}$. Define the function $g_{\alpha}$ by
\[g_{\alpha}(z) = z^{\alpha} e^{-z^{\alpha}}.\]
Then $g$ is in $H^{\infty}_0(\Sigma_{\phi})$ for every $\phi<\pi/(2\alpha)$.
Moreover, $g_{\alpha}(t\textbf{A})x = t^{\alpha} \frac{d}{dt}\textbf{T}_{\alpha}(t)x$. Thus Theorem \ref{thm:estimatingLrLqLP} implies the following:
\begin{corollary}\label{cor:Talpha}
Assume the conditions of Theorem \ref{thm:estimatingLrLqLP} and let $\alpha\in (0,1)$. Then the following estimate holds for every $x\in  L^r(\Omega;X)$:
\begin{align*}
\|t\mapsto g_{\alpha}(t\textbf{A})x\|_{L^r(\Omega;L^q(\R_+,\frac{dt}{t};X))} & \lesssim \|x\|_{L^r(\Omega;X)}\lesssim \|t\mapsto g_{\alpha}(t\textbf{A})x\|_{L^r(\Omega;L^p(\R_+,\frac{dt}{t}))}.
\end{align*}
\end{corollary}

\begin{remark}\
\begin{enumerate}[(i)]
\item The above estimate for $\alpha = 1/2$ was proved for martingale type $p$ cotype $q$ spaces in \cite{MTX06} under the additional conditions that $A$ is selfadjoint and $T(t)$ is a contraction on $L^r(\Omega)$ for all $r\in [1, \infty]$ and $T(t) 1 = 1$ for every $t\geq 0$.
\item Under the additional assumption that $X$ is a complex interpolation space of a Hilbert space and a UMD Banach space, one can prove the assertion of Theorem \ref{thm:estimatingLrLqLP} for every function $f\in H^{\infty}_0(\Sigma_{\phi})$,
    where $\phi<\pi/2$ depends on $X$ and $A$ (see \cite[Theorem 9.7]{Hyt07} and \cite[Theorem 9.3]{KalW04}).
\item If $A$ is not injective, then one can still prove results such as Theorem \ref{thm:estimatingLrLqLP} and Corollary \ref{cor:Talpha}. Indeed, the reflexive space $L^r(\Omega;X)$ is a direct sum of the kernel of $A$ and $\overline{R(A)}$ (see \cite[Appendix A]{KunWeis04}) and the restriction of $A$ to $\overline{R(A)}$ satisfies the required conditions. The estimates of Theorem \ref{thm:estimatingLrLqLP} and Corollary \ref{cor:Talpha} will now hold with $\|x\|_{L^r(\Omega;X)}$ replaced by $\|x-Px\|_{L^r(\Omega;X)}$, where $P$ is the projection onto the kernel of $A$.
\end{enumerate}
\end{remark}

\subsection{Embedding of interpolation spaces}

In this final section we consider embeddings of the form
\begin{align}\label{eq:embedPeetre}
(X_0, X_1)_{\theta,p} \hookrightarrow [X_0, X_1]_{\theta} \hookrightarrow (X_0, X_1)_{\theta,q},
\end{align}
where $1\leq p\leq q\leq \infty$ and $\theta\in (0,1)$. Here $(\cdot, \cdot)_{\theta, r}$ denotes the real interpolation space with parameter $r$ and $[\cdot, \cdot]_{\theta}$ denotes the complex interpolation space. We refer to  \cite{BerLof} and \cite{Tri} for a detailed treatment of the subject. The embedding \eqref{eq:embedPeetre} always holds for $p=1$ and $q=\infty$ (see \cite{Tri}). Under the assumption that $X_j$ have Fourier type $p_j\in [1, 2]$ for $j=0, 1$, in \cite{Peetre69} Peetre has improved the embedding \eqref{eq:embedPeetre} to $p\in [1, 2]$  satisfying $\frac{1-\theta}{p_0} + \frac{\theta}{p_1} = \frac{1}{p}$ and $q = p'$ (also see \cite{HNVW1}).

Recall the following facts:
\begin{itemize}
\item every space has Fourier type $1$.
\item a Hilbert space has Fourier type $2$.
\item Fourier type $p$ implies Fourier type $q$ if $1\leq q\leq p\leq 2$.
\item $L^r$ has Fourier type $\min\{r, r'\}$.
\item Fourier type $p$ implies type $p$ and cotype $p'$.
\end{itemize}
The result of Peetre is optimal in the sense that one cannot take a better value of $p$ and $q$ in general. We will improve Peetre's result in the case $X_1 = D(A)$, where $A$ is a certain sectorial operator on $X$.

Recall that for $\theta\in (0,1)$ and $p\in [1, \infty]$ an equivalent norm on $(X,D(A))_{\theta, p}$ is given by (see \cite[Theorem 1.14.3]{Tri})
\begin{equation}\label{eq:interpspaces}
\begin{aligned}
\|x\|_{(X,D(A))_{\theta, p}}& \eqsim \|x\|+\|t\mapsto t^{-\theta} w(tA)\|_{L^p(\R_+,\frac{dt}{t};X)},
\end{aligned}
\end{equation}
where $w(z) = z(1+z)^{-1}$.
In the case $A$ is invertible then term $\|x\|$ can be omitted from the above expressions.

\begin{theorem}\label{thm:complexreal}
Let $X$ be a Banach space with type $p\in [1, 2]$ and cotype $q\in [2, \infty]$. Assume $A$ has a bounded $H^\infty$-calculus of some angle $\sigma\in (0,\pi)$. Then for all $\theta\in (0,1)$,
\begin{align*}
(X, D(A))_{\theta,p} & \hookrightarrow  D(A^{\theta}) \hookrightarrow  (X, D(A))_{\theta,q},
\\ (X, D(A))_{\theta,p} & \hookrightarrow  [X,D(A)]_{\theta}  \hookrightarrow (X, D(A))_{\theta,q}.
\end{align*}
\end{theorem}

If $A$ has a bounded $H^\infty$-calculus it also has bounded imaginary powers, and therefore $D(A^{\theta}) = [X, D(A)]_{\theta}$ (see \cite[6.6.9]{Haase:2} and \cite[1.15.3]{Tri}). Theorem \ref{thm:complexreal} proves \eqref{eq:embedPeetre} under type and cotype assumptions which in this special but important case improves the result of Peetre (see Example \ref{ex:Besovemdtype}).

We now turn to the proof of Theorem \ref{thm:complexreal}.
Replacing $A$ by $A+1$ if necessary, we may assume that $A$ is invertible and $\|(\lambda-A)^{-1}\|\leq M(1+|\lambda|)^{-1}$ for all $\lambda\notin \Sigma_{\sigma}$. This does not influence the interpolation spaces and fractional domain spaces.

\begin{proof}
As we already noticed that $D(A^{\theta}) = [X, D(A)]_{\theta}$, it suffices to prove the embedding for $D(A^{\theta})$. We first make a general observation. Let $x\in D(A)$. Let $v(z) = z^{1-\theta} (1+z)^{-1}$. Then $v(zA)y = z^{-\theta} w(zA) x$, where $y = A^{\theta} x$ and $w(z) = z(1+z)^{-1}$ is as before. Observe that for $r\in [1, \infty)$, \eqref{eq:interpspaces} yields
\begin{align*}
\|v(tA)y\|_{L^r(\R_+,\frac{dt}{t};X)} \eqsim \|x\|_{D_A(\theta, r)}.
\end{align*}
To prove the assertion of the theorem it suffices to consider $q<\infty$ as in the case $q=\infty$ the result is a special case of \cite[1.15.2]{Tri}. In the remaining cases, by density it suffices to show $\|x\|_{D_A(\theta, q)}\lesssim \|y\|\leq \|x\|_{D_A(\theta, p)}$ for $x\in D(A)$ and $y = A^{\theta}x$ (see \cite[1.6.2 and 1.15]{Tri}).

Combining the above observation with Proposition \ref{prop:generalPL}  we find
\begin{align*}
\|x\|_{D_A(\theta, q)} \eqsim \|v(tA)y\|_{L^q(\R_+,\frac{dt}{t};X)} \lesssim \|y\| \lesssim \|v(tA)y\|_{L^p(\R_+,\frac{dt}{t};X)} \eqsim \|x\|_{D_A(\theta, p)}.
\end{align*}
\end{proof}

As an illustration we apply Theorem \ref{thm:complexreal} to the case of Sobolev spaces and compare the obtained embedding with the result one would get if Fourier type is used instead.
\begin{example}\label{ex:Besovemdtype}
Let $X = L^r(\R^d)$ with $r\in (1, \infty)$ and $A = \Delta$ with $D(A) = W^{2,r}(\R^d)$. Then $A$ has a bounded $H^\infty$-calculus. Fix $\theta\in (0,1)$. It follows from \cite[2.4.2]{Tri} that $D(A^{\theta}) = [X, D(A)]_{\theta} =  H^{r,2\theta}(\R^d)$ and $(X, D(A))_{\theta,p} = B^{2\theta}_{r,p}(\R^d)$.
\begin{enumerate}[(i)]
\item In the case $r\in [2, \infty)$, $X$ has type $2$ and cotype $r$ and Theorem \ref{thm:complexreal}  yields
\begin{equation}\label{eq:Besov1}
B^{2\theta}_{r,2}(\R^d)\hookrightarrow H^{r,2\theta}(\R^d) \hookrightarrow B^{2\theta}_{r,r}(\R^d).
\end{equation}
It is known that the microscopic coefficients $2$ and $r$ cannot be improved.
\item In the case $r\in (1,2]$, $X$ has type $r$ and cotype $2$ and Theorem \ref{thm:complexreal}  yields
\begin{equation}\label{eq:Besov2}
B^{2\theta}_{r,r}(\R^d)\hookrightarrow H^{r,2\theta}(\R^d) \hookrightarrow B^{2\theta}_{r,2}(\R^d).
\end{equation}
Also in this case it is known that the microscopic parameters $r$ and $2$ cannot be improved.
\end{enumerate}
If instead one uses Fourier type, one only obtains $B^{2\theta}_{r,r'}(\R^d)\hookrightarrow H^{r,2\theta}(\R^d)$ on the left-hand side of \eqref{eq:Besov1} and $H^{r,2\theta}(\R^d) \hookrightarrow B^{2\theta}_{r,r'}(\R^d)$ on the right-hand side of \eqref{eq:Besov2}.
\end{example}

The results of this section lead to the following natural question:
\begin{problem}\label{probl:complex}
Given an interpolation couple $(X_0,X_1)$ and $p\in (1, 2]$ and $q\in [2, \infty)$.
Prove or disprove the following:
\begin{enumerate}[(i)]
\item If $X_0$ and $X_1$ both have type $p\in (1, 2]$, then $(X_0, X_1)_{\theta,p}\hookrightarrow [X_0, X_1]_{\theta}$.
\item If $X_0$ and $X_1$ both have cotype $q\in [2, \infty)$, then $[X_0, X_1]_{\theta}\hookrightarrow (X_0, X_1)_{\theta,q}$.
\end{enumerate}
\end{problem}
More generally, one can ask for the same result if instead $X_j$ has type $p_j$ and $\frac{1}{p} = \frac{1-\theta}{p_0} + \frac{\theta}{p_1}$ (and similarly in the cotype case).

\begin{remark}
Replacing the complex interpolation method (see \cite{KKW,SuWe06}) by the so-called Rademacher interpolation method $\lb \cdot, \cdot\rb_{\theta}$, one can prove the embedding in Problem~\ref{probl:complex} if $X_j$ has type $p_j$ and cotype $q_j$ and $\frac{1-\theta}{p_0} + \frac{\theta}{p_1} = \frac{1}{p}$ and $\frac{1-\theta}{q_0} + \frac{\theta}{q_1} = \frac{1}{q}$.
The Rademacher interpolation method differs from the complex method in general. Indeed, for an almost $R$-sectorial operator $A$ it is known that if $D(A^{\theta}) = \lb X, D(A)\rb_{\theta}$ and $X$ is $B$-convex (nontrivial type), then $A$ has an $H^\infty$-calculus (see \cite[Corollary 7.7]{KKW}). Since there exists an almost $R$-sectorial operator $A$ on $L^p(\R)$ with bounded imaginary powers but without a bounded $H^\infty$-calculus (see \cite[Example 10.17]{KunWeis04}), it follows that for this operator
\[\lb X, D(A)\rb_{\theta} \neq D(A^{\theta}) = [X, D(A)]_{\theta},\]
where the last identity follows from \cite[1.15.3]{Tri}.

On the other hand, in the case $A$ has an $H^\infty$-calculus the Rademacher interpolation and complex method indeed coincide (see \cite[Theorem~7.4]{KKW}), and therefore, Theorem \ref{thm:complexreal} can alternatively be derived from \cite[Theorem~6.1]{SuWe06}.
\end{remark}

\def\cprime{$'$}

\end{document}